\documentclass[12pt]{amsart}
\usepackage{xspace,amssymb,amsfonts,epsfig,marvosym,euscript,eufrak,xypic,enumerate,amsmath,nicefrac}
\usepackage{tikz,etex,bbm,xspace}
 \usepackage{epstopdf,ifpdf}
\usepackage{multicol,microtype}
\usepackage[margin=3cm,footskip=25pt,headheight=20pt]{geometry}
\usepackage{pxfonts,fancyhdr}
\usepackage{pstricks}
\usepackage{pst-tree, pst-node}
\ifpdf
  \usepackage{pdfsync}
\fi

\newcommand{\ol}{\overline}
\newcommand{\nc}{\newcommand}
\newcommand{\mc}{\mathcal}

\nc{\on}{\operatorname}
\nc{\h}{\mathfrak{h}}
\nc{\g}{\mathfrak{g}}
\nc{\n}{\mathfrak{n}}
\nc{\ch}{\on{CH}}
\nc{\wt}{\widetilde}

\nc{\F}{\mc{F}}
\nc{\C}{\mc{C}}
\nc{\M}{\mc{M}}
\nc{\T}{\mc{T}}
\renewcommand{\H}{\on{H}}
\nc{\G}{\mc{G}}
\nc{\Gb}{\overline{G}}

\nc{\VFun}{\on{Vect}(\fun)}

\nc{\FF}{\mathbb{F}}
\nc{\HH}{\mathbb{H}}

\nc{\Amod}{$A$-{mod}}

\renewcommand{\H}{\mathbb{H}}

\nc{\Pone}{\mathbb{P}^1}
\nc{\Aone}{\mathbb{A}^1}
\nc{\fun}{\mathbb{F}_1}
\nc{\mf}{\mathfrak}
\nc{\slthat}{\widehat{\mf{sl}}_2}

\nc{\p}{\mathfrak{p}}
\nc{\spec}{\on{Spec}}
\nc{\Msch}{\mc{M}sch}

\nc{\fm}{\langle t \rangle}
\nc{\A}{\on{A}}

\theoremstyle{definition}

\newtheorem{theorem}{Theorem}
\newtheorem{proposition}{Proposition}
\newtheorem{lemma}{Lemma}
\newtheorem{definition}{Definition}

\newtheorem{remark}{Remark}

\theoremstyle{definition}
\newtheorem{example}{Example}

\begin{document}

\title{On the Hall algebra of semigroup representations \\ over $\fun$}
\author{Matt Szczesny} 
\address{Department of Mathematics  and Statistics, 
         Boston University, Boston MA, USA}
\email{szczesny@math.bu.edu}

\begin{abstract}
Let $\A$ be a finitely generated semigroup with $0$. An $\A$--module over $\fun$ (also called an $\A$--set), is a pointed set $(M,*)$ together with an action of $\A$. We define and study the Hall algebra $\H_{\A}$ of the category $\C_{\A}$ of finite $\A$--modules. $\H_{\A}$ is shown to be the universal enveloping algebra of a Lie algebra $\n_{\A}$, called the \emph{Hall Lie algebra} of $\C_{\A}$.  In the case of the $\fm$ - the free monoid on one generator $\fm$, the Hall algebra (or more precisely the Hall algebra of the subcategory of nilpotent $\fm$-modules) is isomorphic to Kreimer's Hopf algebra of rooted forests. This perspective allows us to define two new commutative operations on rooted forests. We also consider the examples when $\A$ is a quotient of $\fm$ by a congruence, and the monoid $G \cup \{ 0\}$ for a finite group $G$. 
\end{abstract}
\maketitle

\medskip

\section{introduction}

The aim of this paper is to define and study the Hall algebra of the category of set-theoretic representations of a semigroup. Classically, Hall algebras have been studied in the context abelian categories linear over finite fields $\mathbb{F}_q$. Given such a category $\mc{A}$, \emph{finitary} in the sense that $\on{Hom}(M,N)$ and $\on{Ext}^{1}(M,N)$ are finite-dimensional $\forall \; M,N \in \mc{A}$ (and therefore finite sets), we may construct from $\mc{A}$ an associative algebra $\HH_{\mc{A}}$ defined over $\mathbb{Z}$ \footnote{It is common to include a twist which makes this algebra over $\mathbb{Q}(\nu)$, where $\nu^2 = q$ }, called the Ringel-Hall algebra of $\mc{A}$. As a $\mathbb{Z}$--module, $\HH_{\mc{A}}$ is freely generated by the isomorphism classes of objects in $\mc{A}$, and its structure constants are expressed in terms of the number of extensions between objects.  Explicitly, if $\ol{M}$ and $\ol{N}$ denote two isomorphism classes, their product in $\HH_{\mc{A}}$ is given by 
\begin{equation} \label{hallmult}
\ol{M} \star \ol{N} =  \sum_{ \ol{R} \in \on{Iso}(\mc{A}) } \mathbb{P}^R_{M,N}  \; \ol{R}
\end{equation}
where $\on{Iso}(\mc{A})$ denotes the set of isomorphism classes in $\mc{A}$, and
\begin{equation} \label{sc}
\mathbb{P}^R_{M,N} = \# \vert \{ L \subset R, L \simeq N, R/L \simeq M \} \vert
\end{equation}
Denoting by $\on{Aut}(M)$ the automorphism group of $M \in \on{Iso}(\mc{A})$, it is easy to see that
\[
\mathbb{P}^R_{M,N} \vert \on{Aut}(M) \vert \vert \on{Aut}(N) \vert
\]
counts the number of short exact sequences of the form
\begin{equation} \label{ses}
0 \rightarrow N \rightarrow R \rightarrow M \rightarrow 0,
\end{equation}
showing that $\H_{\mc{A}}$ encodes the structure of extensions in $\mc{A}$. Under additional assumptions on $\mc{A}$, $\HH_{\mc{A}}$ can be given the structure of a Hopf algebra (see \cite{S}). 

A closer examination of formula (\ref{hallmult}) and the description of $\mathbb{P}^R_{M,N}$ in terms of short exact sequences (\ref{ses}) reveals that it makes sense in situations where $\mc{A}$ is not abelian, or even additive. It suffices that
 $\mc{A}$ be an exact category in the sense of Quillen, satisfying certain finiteness conditions (see \cite{Hub}), or possibly a non-additive analogue thereof (see \cite{DK} for a very general framework).   One such example is the category of set-theoretic representations of a semigroup (for other examples of Hall algebras in a non-additive context see \cite{KS, Sz1, Sz2, Sz3}). Given a finitely generated semigroup $A$ possessing and absorbing element $0$, we define a (finite) \emph{$\A$--module} to be a finite pointed set $(M,*)$ equipped with an action of $\A$. Maps between $\A$-modules are defined to be maps of pointed sets compatible with the action of $\A$, and we denote the resulting category by $\C_{\A}$. $\C_{\A}$ possesses many of the good properties of an abelian category, such as the existence of zero object, small limits and co-limits, and in particular kernels and co-kernels. We may therefore talk about short exact sequences in $\C_{\A}$.

 At the same time, $\C_{\A}$ differs crucially from an abelian category in that it is not additive, and morphisms $f: M \rightarrow N$ are not necessarily \emph{normal}, meaning that the natural map $Im(f) \rightarrow coim(f)$ is not in general an isomorphism. This means that the correspondence between the definition of $\mathbb{P}^R_{M,N}$ in (\ref{sc}) and the count of short exact sequences breaks down. In fact, given $M, N$ in $\C_{\A}$, there will in general be infinitely many distinct short exact sequences of the form (\ref{ses}). We may however pass to the subcategory $\C^N_{\A}$ of $\C_{\A}$ consisting of the same objects, but with only normal morphisms. The requirement that $f$ be normal is easily seen to be equivalent to the property that the fiber $f^{-1}(n)$  over an element $n \in N$ contain at most one element with the exception of $f^{-1}(*)$. If $L \subset R$ is a sub-module, then all morphisms in the short exact sequence
 \[
 0 \rightarrow L \rightarrow R \rightarrow R/L \rightarrow 0
 \]
 are normal, and conversely, any normal short exact sequence corresponds to the data in (\ref{sc}) (up to automorphisms of the kernel and co-kernel). Thus, in $\C^N_{\A}$, the correspondence between the two descriptions of $\mathbb{P}^R_{M,N}$ which holds in the abelian setting is restored. Furthermore, in $\C^N_{\A}$ there are only finitely many extensions between any two objects, and so we may define $\H_{\C^N_{\A}}$ as in the abelian case. $\H_{\C^N_{\A}}$ may be further equipped with a co-commutative co-multiplication and antipode (after tensoring with $\mathbb{Q}$), which gives it the structure of a graded connected co-commutative Hopf algebra. The Milnor-Moore theorem shows that $\H_{\C^N_{\A}}$ is isomorphic to the enveloping algebra $\mathbb{U}(\n_{\C^N_{\A}})$ of the Lie algebra $\n_{\C^N_{\A}}$ of its primitive elements, which correspond to indecomposable $A$-modules.  To summarize:
 
 \bigskip
 \begin{theorem}
 There is an isomorphism of Hopf algebras $\H_{\C^{N}_{\A}} \simeq \mathbb{U}(\n_{\C^N_{\A}})$, where $\H_{\C^{N}_{\A}}$ is the Hall algebra of the category $\C^N_{\A}$ of finite $\A$--modules with normal morphisms, and $\n_{\C^N_{\A}}$ is the Lie sub-algebra spanned by indecomposable $A$--modules.
 \end{theorem}
 \bigskip

This construction may be re-cast in the yoga of the "field with one element", denoted $\fun$ (for more on $\fun$, see \cite{KapS, H, M, D, D2, CLS, CC1, CC2, Sou, Tv}). The basic principle in working with $\fun$ is that one loses any additive structure and only multiplication remains. Semigroups are therefore analogues of (possibly non-unital) algebras over $\fun$, monoids analogues of unital algebra, and pointed sets analogues of vector spaces. The study of semigroup actions on pointed sets is therefore the $\fun$--analogue of the representation theory of algebras over a field. The above result shows that the category of representations of "an algebra over $\fun$" leads, via the Hall algebra construction, to a Hopf algebra in a way analogous to the situation over $\mathbb{F}_q$. 

The study of $\H_{\C^N_A}$ turns out to be interesting already in the case when $A$ is the free monoid $\fm$ on one generator, i.e.
\[
\fm := \{ 0, 1, t, t^2, t^3, \cdots \}.
\]
The study of $\fm$--modules may be seen as linear algebra over $\fun$, since given a field $k$, the monoid algebra over $k$ is the polynomial ring $k[t]$. To a $\fm$--module $M$ we may attach a directed graph $\Gamma_M$ whose vertices correspond to (non-zero) elements of $M$, and whose directed edges are $\{ m \rightarrow t \cdot m \}$. We give a description of the possible graphs, and identify the nilpotent $\fm$--modules with rooted forests. The latter form a full subcategory $\C^{N}_{\fm, nil}$ of $\C^N_{\fm}$, and we show that the Hall algebra $\H^{N}_{\fm, nil}$ is isomorphic to the dual of Kreimer's Hopf algebra of rooted trees (\cite{Kre}), which encodes the combinatorial structure of perturbative renormalization in quantum field theory. To summarize:

\bigskip
\begin{theorem}
There is a Hopf algebra isomorphism $\H_{\C^N_{\fm,nil}} \simeq \H^*_K$ where $\H_{\C^N_{\fm,nil}}$ is the Hall algebra of the category of finite nilpotent $\fm$--modules, and $\H_K$ is Kreimer's Hopf algebra of rooted trees. 
\end{theorem}
\bigskip

The monoid $\fm$ is the path monoid of the Jordan quiver, and so this result may also be interpreted in the context of quiver representations over $\fun$.  It is worth remarking that over a finite field $\mathbb{F}_q$, the Hall algebra of nilpotent representations of the Jordan quiver is isomorphic to the Hopf algebra of symmetric functions $\Lambda$ (see \cite{S}). $\H^*_K$ is therefore an $\fun$ analogue of $\Lambda$. 

This paper is organized as follows. In section \ref{semigroups} we recall basic facts about semigroups, monoids, and the category $\C_{\A}$ of their set-theoretic representations, as well as the normal sub-category $\C^N_{\A}$, which is better adapted to the Hall algebra construction. 
We also define "representation rings" $\on{Rep}^{\wedge}(\A), \on{Rep}^{\otimes}(\A)$. Section \ref{Hall_Alg} is devoted to the definition of the Hall algebra $\H_{\C^N_{\A}}$ and proof of Theorem 1. Section \ref{examples} contains the main examples. In \ref{fm} we consider the case of the free monoid $\fm$, and relate $\H_{\C^N_{\fm}}$ to Kreimer's Hopf algebra of rooted trees. Section \ref{fm_cong} we consider the case of quotients $\fm/ \sim$ by a congruence. Finally, \ref{G} looks at the case of the monoid ${G} \cup \{ 0 \}$, where $G$ is a finite group. 

\bigskip

\noindent {\bf Acknowledgements:} I would like to thank Oliver Lorscheid and Anton Deitmar for valuable conversations. I'm especially grateful to Dirk Kreimer for his hospitality during my visit to HU Berlin where part of this paper was written, as well as many interesting ideas.  

\setcounter{theorem}{0}
\section{semigroups, monoids, and their modules} \label{semigroups}

In this section, we recall some basic properties related to semigroups, monoids, and their (set-theoretic) representations. Good references are \cite{KKM, CLS}. We will quite often represent semigroup actions by graphs, and so start by fixing notation. Given a (possibly directed) graph $\Gamma$, we denote by $V(\Gamma)$ its vertex set, and by $E(\Gamma)$ its edge set. For a directed edge $e \in E(\Gamma)$, let $s(e)$ and  $t(e)$ denote its initial and terminal vertices respectively. 

\begin{definition}
By a \emph{semigroup} we will always mean a multiplicatively written semigroup $\A$ with absorbing element $0$, satisfying
\[
0 \cdot a = a \cdot 0 = 0  \; \forall a \in \A.
\]
A \emph{monoid} is semigroup $\A$ together with identity element $1$, satisfying
\[
1 \cdot a = a \cdot 1 = a \; \; \forall a \in \A.
\]
A \emph{morphism} $f: \A \rightarrow \on{B} $ of semigroups is a multiplicative map preserving $0$. If $\A$ and $\on{B}$ are monoids, we require the map to preserve $1$ as well.  
\end{definition}

\noindent We denote the category of semigroups by $\M_0$, and view it as the $\fun$--analogue of the category of associative  (but not necessarily unital) algebras. Given a general semigroup $\on{B}$ not necessarily possessing $1$, we may adjoin to it an identity obtain a monoid $\A = \on{B} \cup \{1\}$ in the above sense (by the same procedure, we may adjoin to a general semigroup not possessing an absorbing element a $0$).  We say that $\A$ is \emph{finitely generated} if there exists a finite collection $\{ a_1, \cdots, a_n \} \in \A$ such that every element of $\A$ can be written as word in $a_1, \cdots, a_n$. 

\bigskip

\noindent{Examples:}

\begin{enumerate}
\item The free semigroup $\langle x_1, \cdots, x_n \rangle$ on $n$ generators, which consists of $0$ and all words in the letters $x_1, \cdots, x_n$ under concatenation.
\item The monoid $\fun = \{ 0, 1 \}$, sometimes called the \emph{field with one element}. 
\item The free commutative monoid on one generator $\fm = \{ 0, 1, t, t^2, t^3, \cdots \}$
\item Any group $G$ is automatically a semigroup. We obtain a monoid $\overline{G} = G \cup \{ 0 \}$ by adjoining $0$. 
\item For a ring $R$, we obtain its multiplicative monoid $R^{\times}$ by forgetting the additive structure. 
\end{enumerate}
 
Given a ring $R$, there exists a base-change functor
\begin{equation} \label{mb_change}
\otimes_{\fun} R : \M_0 \rightarrow R-alg
\end{equation}
to the category of $R$--algebras, defined by setting
\[
A \otimes_{\fun} R := R[\A] 
\]
where $R[A]$ is the monoid algebra:
\[
R[\A] := \left\{ \sum r_i a_i | a_i \in \A, r_i \in R \right\}/ \langle 0 \rangle
\]
with multiplication induced from the monoid multiplication. 

\medskip
\begin{definition}
A \emph{congruence} on a semigroup $\A$ is an equivalence relation $\sim$ on $\A$ such that if $x,y,u,v \in \A$ and $x \sim y,  u \sim v$,  then $xu \sim yv$. We denote by $\ol{x}$ the image of $x \in \A$ in $\A / \sim$.  $\A/\sim$ inherits a semigroup structure with $\ol{x} \cdot  \ol{y} := \ol{xy}$. 
\end{definition}

\medskip

\begin{example} \label{cong_ex}
Consider the free monoid $\fm$. For any $x \in \fm, n \in \mathbb{N}$, the equivalence relation generated by $t^{k+n} \sim t^k x, \; k \geq 0$ is a congruence on $\fm$. It is easy to see (see section \ref{fm_cong}) that every congruence on $\fm$ is of this form. 
\end{example}

\subsection{$\A$--modules} \label{Amodules}

\begin{definition}
Let $\A$ be a monoid. An \emph{$\A$--module} is a pointed set $(M,*)$ equipped with an action of $\A$. More explicitly, an $\A$--module structure on $(M.*)$ is given by a map
\begin{align*}
\A \times M  & \rightarrow M \\
(a, m) & \rightarrow a \cdot m
\end{align*}
satisfying 
\[
(a \cdot b)\cdot m = a \cdot (b \cdot m), \; \; \; 1 \cdot m = m, \; \; \; 0 \cdot m = *, \; \; \forall a,b, \in \A, \; m \in M
\]
\end{definition}
A \emph{morphism} of $\A$--modules is given by a pointed map $f: M \rightarrow N$ compatible with the action of $\A$, i.e. $f(a \cdot m) = a \cdot f(m)$. 
The $\A$--module $M$ is said to be \emph{finite} if $M$ is a finite set, in which case we define its \emph{dimension} to be $dim(M) = \vert M \vert -1$ (we do not count the basepoint, since it is the analogue of $0$). We say that $N \subset M$ is an \emph{$\A$--submodule} if it is a (necessarily pointed) subset of $M$ preserved by the action of $\A$. $\A$ always posses the trivial module $\{*\}$, which will be referred to as the \emph{zero module}.
\medskip

\noindent {\bf Note:} This structure is called an $\A$--act in \cite{KKM} and an \emph{$\A$--set} in \cite{CLS}. 

\medskip

We denote by $\C_{\A}$ the category of finite $\A$--modules. It is the $\fun$ analogue of the category of a finite-dimensional representations of a an algebra. In particular, a $\fun$--module is simply a pointed set, and will be referred to as a vector space over $\fun$. 

\medskip

Given a morphism $f: M \rightarrow N$ in $\C_{\A}$, we define the \emph{image} of $f$ to be $$Im(f) := \{ n \in N \vert \exists m \in M, f(m) = n \}.$$ For $M \in \C_{\A}$ and an $\A$--submodule $N \subset M$, the \emph{quotient} of $M$ by $N$, denoted $M/N$ is the $\A$--module $$ M/N :=  M \backslash N \cup \{* \}, $$ i.e. the pointed set obtained by identifying all elements of $N$ with the base-point, equipped with the induced $\A$--action. 

\bigskip

We recall some properties of $\C_{\A}$,  following \cite{KKM, CLS}, where we refer the reader for details:

\medskip

\begin{enumerate}
\item The trivial $\A$--module $0 = \{ * \}$ is an initial, terminal, and hence zero object of $\C_{\A}$. 
\item Every morphism $f: M \rightarrow N$ in $C_A$ has a kernel $Ker(f):=f^{-1}(*)$.
\item  Every morphism $f: M \rightarrow N$ in $C_A$ has a cokernel $Coker(f):=M/Im(f)$. 
\item The co-product of a finite collection $\{ M_i \}, i \in I$ in $\C_{\A}$ exists, and is given by the wedge product
$$
\bigvee_{i \in I} M_i = \coprod M_i / \sim
$$
where $\sim$ is the equivalence relation identifying the basepoints. We will denote the co-product of $\{M_i \}$ by $$\oplus_{i \in I} M_i$$
\item The product of a finite collection $\{ M_i \}, i \in I$ in $\C_{\A}$ exists, and is given by the Cartesian product $\prod M_i$, equipped with the diagonal $\A$--action. It is clearly associative. it is however not compatible with the coproduct in the sense that $M \times (N \oplus L) \nsimeq M \times N \oplus M \times L$.
\item The category $\C_{\A}$ possesses a reduced version $M \wedge N$ of the Cartesian product $M \times N$, called the smash product. $M \wedge N := M \times N / M \vee N$, where $M$ and $N$ are identified with the $\A$--submodules $\{ (m,*) \}$ and $\{(*,n)\}$ of $M \times N$ respectively. The smash product inherits the associativity from the Cartesian product, and is compatible with the co-product - i.e. $M \wedge (N \oplus L) \simeq M \wedge N \oplus M \wedge L$.
\item $\C_{\A}$ possesses small limits and co-limits. 
\item If $\A$ is commutative, $\C_{\A}$ acquires a monoidal structure called the \emph{tensor product}, denoted $M \otimes_{\A} N$, and defined by
\[
M \otimes_{\A} N := M \times N / \sim_{\otimes}
\]
where $\sim_{\otimes}$ is the equivalence relation generated by $(a \cdot m, n) \sim_{\otimes} (m, a \cdot n)$ for all $a \in \A, m \in M, n \in N$. Note that $(*,n) = (0 \cdot *, n) \sim_{\otimes} (*, 0 \cdot n) = (*,*)$, and likewise $(m,*) \sim_{\otimes} (*,*)$. This allows us to rewrite the tensor product as $M \otimes_{\A} N = M \wedge N/ \sim_{\otimes'}$, where $\sim_{\otimes'}$ denotes the equivalence relation induced on $M \wedge N$ by $\sim_{\otimes}$. We have 
\begin{align*}
M \otimes_{\A} N  & \simeq N \otimes_{\A} M, \\  (M \otimes_{\A} N) \otimes_{\A} L & \simeq M \otimes_{\A} (N \otimes_{\A} L), \\ \ M \otimes_{\A} (L \oplus N) & \simeq (M \otimes_{\A} L) \oplus (M \otimes_{\A} N). 
\end{align*}
\item Given $M$ in $\C_{\A}$ and $N \subset M$, there is an inclusion-preserving correspondence between flags $N \subset L \subset M$ in $\C_{\A}$ and $\A$--submodules of $M/N$ given by sending $L$ to $L/N$. The inverse correspondence is given by sending $K \subset M/N$ to $\pi^{-1} (K)$, where $\pi: M \rightarrow M/N$ is the canonical projection. This correspondence has the property that if $N \subset L \subset L' \subset M$, then $(L'/N)/(L/N) \simeq L'/L$.  \label{property9}
\end{enumerate}

\medskip

\noindent These properties suggest that $\C_{\A}$ has many of the properties of an abelian category, without being additive. It is an example of a \emph{quasi-exact} and \emph{belian} category in the sense of Deitmar and a \emph{proto-exact} category in the sense of Dyckerhof-Kapranov. Let $\on{Iso}(\C_{\A})$ denote the set of isomorphism classes in $\C_{\A}$, and by $\ol{M}$ the isomorphism class of $M \in \C_{\A}$.

\bigskip

\begin{definition}

\begin{enumerate}
\item We say that $M \in \C_{\A}$ is \emph{indecomposable} if it cannot be written as $M = N \oplus L$ for non-zero $N, L \in \C_{\A}$. 
\item We say $M \in \C_{\A}$ is \emph{irreducible} or \emph{simple} if it contains no proper sub-modules (i.e those different from $*$ and $M$). 
\end{enumerate}
\end{definition}

It is clear that every irreducible module is indecomposable. 
We have the following analogue of the Krull-Schmidt theorem:

\bigskip

\begin{proposition}
Every $M \in \C_{\A}$ can be uniquely decomposed (up to reordering) as a direct sum of indecomposable $\A$--modules. 
\end{proposition}

\begin{proof}
Let $\Omega_M$ be the directed colored graph with vertex set $ V(\Omega_M) := M \backslash *$, and edge set  $$ E(\Omega_M) := \{ m \rightarrow a \cdot m \vert a \in A, a \cdot m \neq * \},$$ where the edge $m \rightarrow a \cdot m$ is colored by $a$. It is clear that
\begin{equation} \label{disj_union}
\Omega_{M \oplus N} = \Omega_M \coprod \Omega_N
\end{equation}
 Let $\Omega_M = \Gamma_1 \cup \Gamma_2 \cdots \cup \Gamma_k$ be the decomposition of $\Omega_M$ into connected components, and $M_{\Gamma_i}$ the subset of $M$ corresponding to $\Gamma_i$, together with the basepoint $*$. Then $M = \oplus M_{\Gamma_k}$, and it is clear  from \ref{disj_union}  that each $M_{\Gamma_i}$ is indecomposable. The uniqueness of the decomposition is immediate. 
\end{proof}

\bigskip

\begin{remark} \label{subobj}

Suppose $M = \oplus^k_{i=1} M_i$ is the decomposition of an $\A$--module into indecomposables, and $N \subset M$ is a submodule. It then immediately follows that $N = \oplus (N \cap M_i)$. 

\end{remark}

\bigskip

Let $\on{Rep}(\A) := \mathbb{Z}[ \ol{M} ]/I ,\; \ol{M} \in \on{Iso}(\C_{\A}) $, where $I$ is the sub-group generated by differences $\ol{M \oplus N} - \ol{M} - \ol{N}$. The fact that the symmetric monoidal operations $\wedge, \otimes$ (when defined) are compatible with $\oplus$ shows that they descend to $\on{Rep}(\A)$. More precisely:

\begin{definition}
Let $\A$ be a semigroup. 
\begin{enumerate}
\item Let $\on{Rep}^{\wedge}(\A)$ denote the commutative ring obtained from  $\on{Rep}(\A)$ using the product induced by the smash product on $\C_{\A}$.
\item If $\A$ is commutative, let $\on{Rep}^{\otimes}(\A)$ denote the commutative ring obtained from $\on{Rep}(\A)$ using the product induced by the tensor product on $\C_{\A}$.
\end{enumerate}
\end{definition}

\bigskip

Given a ring $R$, we obtain a base-change functor (which by abuse of of notation, we will denote by the same symbol as \ref{mb_change})
\begin{equation}
\otimes_{\fun} R : \C_{\A} \rightarrow R[\A]-mod
\end{equation}
to the category of $R[\A]$--modules defined by setting
\[
M \otimes_{\fun} R := \bigoplus_{m \in M, m \neq *} R \cdot m 
\]
i.e. the free $R$--module on the non-zero elements of $M$, with the $R[\A]$--action induced from the $\A$--action on $M$.

\subsection{Exact Sequences}

A diagram $M_1 \overset{f}{\rightarrow} M_2 \overset{g}{\rightarrow} M_3 $ in $\C_{\A}$ is said to be \emph{exact at $M_2$} if $Ker(g) = Im(f)$. A sequence of the form $$ 0 \rightarrow M_1 \rightarrow M_2 \rightarrow M_3 \rightarrow 0 $$ is a \emph{short exact sequence} if it is exact at $M_1, M_2$ and $M_3$. 

One key respect in which $\C_{\A}$ differs from an abelian category is he fact that given a morphism $f: M \rightarrow N$, the induced morphism $M/Ker(f) \rightarrow Im(f)$ need not be an isomorphism. This defect also manifests itself in the fact that the base change functor $\otimes_{\fun} R : \C_{\A} \rightarrow R[\A] -- mod$ fails to be exact. $\C_{\A}$ does however contain a (non-full) subcategory which is well-behaved in this sense, and which we proceed to describe. 

\begin{definition}
A morphism $f: M \rightarrow N$ is \emph{normal} if every fibre of $f$ contains at most one element, except for the fibre $f^{-1}(*)$ of the basepoint $* \in N$. 
\end{definition}

It is straightforward to verify that this condition is equivalent to the requirement that the canonical morphism $M/Ker(f) \rightarrow Im(f)$ be an isomorphism, and that the composition of normal morphisms is normal. 

\bigskip

\begin{definition}
Let $\C^N_{\A}$ denote the subcategory of $\C_{\A}$ with the same objects as $\C_{\A}$, but whose morphisms are restricted to the normal morphisms of $\C_{\A}$. A short exact sequence in $\C^N_{\A}$ is called \emph{admissible}.
\end{definition}

\bigskip

\begin{remark} In contrast to $\C_{\A}$, $\C^N_{\A}$ is typically neither (small) complete nor co-complete. However, $\otimes_{\fun} R$ is exact on $\C^N_{\A}$ for any ring $R$. Note that $\on{Iso}(\C_{\A}) = \on{Iso}(\C^N_{\A})$, since all isomorphisms are normal. 
\end{remark}

\bigskip

\begin{lemma} \label{finitary} Let $\A$ be a semigroup and $\C^N_{\A}$ as above.
\begin{enumerate}
\item For $M,N \in \C^N_{\A}$, $\vert Hom_{\C^N_{\A}}(M,N) \vert < \infty$ \label{part1}
\item Suppose $\A$ is finitely generated. For $M,N \in \C^N_{\A}$, there are finitely many admissible short exact sequences 
\begin{equation} \label{ses2}
0 \rightarrow M \overset{f}{\rightarrow} L \overset{g}{\rightarrow} N \rightarrow 0
\end{equation}
up to isomorphism. 
\end{enumerate}
\end{lemma}

\begin{proof}
\begin{enumerate}
\item This is obvious, since $M,N$ are finite sets. 
\item Let $n \in \mathbb{N}$. We begin by showing that up to isomorphism, there are finitely many $K \in \C^N_{\A}$ such that $dim(K) =n$. The action of $\A$ on $K$ can be specified by giving the action of each of the generators, which corresponds to an element of $\on{Hom}_{pSet}(K,K)$. The latter is is a finite set, and so the claim follows. Since $dim(L) = dim(M) + dim(N)$, it follows that $L$ can belong to only finitely many isomorphism classes in $\C^N_{\A}$. The result now follows from (\ref{part1}).
\end{enumerate}
\end{proof}

\bigskip

\begin{remark}
The requirement that $\A$ be finitely generated in part $(2)$ of Lemma \ref{finitary} is necessary. Let $\langle x_{\mathbb{N}} \rangle := \langle x_1, x_2, \cdots \rangle$ denote the free semigroup on countably many generators $x_1, x_2, \cdots$, and let $V$ denote the $\langle x_{\mathbb{N}} \rangle $--module whose underlying pointed set is $\{ e, * \}$ (i.e. it contains one non-zero element, so that $\vert V \vert = 1$), and $x_i \cdot e = *, \forall i \in \mathbb{N}$. If $$ 0 \rightarrow V \rightarrow W \rightarrow V \rightarrow 0$$ is an admissible short exact sequence, then we can identify $W$ with the pointed set $\{ e, f, * \}$, where $f$ maps to the nonzero element in the cokernel. We have $x_i \cdot e = *$, but we may freely choose $x_i \cdot f = e$ or $x_i \cdot f = *$. There are therefore infinitely many mutually non-isomorphic choices for $W$. 
\end{remark}

\subsubsection{The Grothendieck group}

We may use the category $\C^N_{\A}$ to attach to $\A$ an invariant $K_0 (\A)$ - the Grothendieck group of $\C^N_{\A}$. Let $$ K_0 (\A) := \mathbb{Z}[\overline{M}] / J , \; \; \;  \ol{M} \in \on{Iso}(\C^N_{\A}), $$ where $J$ is the subgroup generated by $\ol{L} - \ol{M} - \ol{N}$ for all admissible short exact sequences (\ref{ses2}).

\section{The Hall algebra of $\C^N_{\A}$} \label{Hall_Alg}

Let $\A$ be a finitely generated semigroup. In this section, we define the Hall algebra of the category $\C^N_{\A}$. In order to off-load notation, we will denote it by $\H_{\A}$ rather than the more cumbersome $\H_{\C^N_{\A}}$ used in the introduction. For more on Hall algebras see \cite{S}.

\medskip
 
As a vector space:
\begin{equation*} 
\H_{\A} := \{ f: \on{Iso}(\C^N_{\A}) \rightarrow \mathbb{Q} \; \vert \; \# \on{supp} (f) < \infty \}.
\end{equation*}
We equip $\H_{\A}$ with the convolution product
\begin{equation*} \label{Hall_prod}
f \star g (\ol{M}) = \sum_{N \subset M} f(\ol{M/N}) g(\ol{N}),
\end{equation*}
where the sum is over all $\A$ sub-modules $N$ of $M$ (in what follows, it is conceptually helpful to fix a representative of each isomorphism class). 
Note that Lemma \ref{finitary} and the finiteness of the support of $f,g$ ensures that the sum in 
(\ref{Hall_prod}) is finite, and that $f \star g$ is again finitely supported.

\begin{lemma}
The convolution product $*$ is associative. 
\end{lemma}

\begin{proof}
Suppose $f,g,h \in \H_{\A}$. Then
\begin{align*}
(f \star (g \star h)) (\ol{M}) &= \sum_{N \subset M} f(\ol{M/N}) (g \star h)(\ol{N}) \\
                                   &= \sum_{N \subset M} f(\ol{M/ N}) (\sum_{L \subset N} g(\ol{N/L}) h(\ol{L}))  \\
                                   &= \sum_{L \subset N \subset M} f(\ol{M/ N}) g(\ol{N/L}) h(\ol{L})
\end{align*}
whereas
\begin{align*}
((f \star g) \star h ) (M) &= \sum_{L \subset M} (f \star g) (\ol{M / L}) h(\ol{L}) \\
                                    &= \sum_{L \subset M} (\sum_{K \subset M / L} f(\ol{(M / L) / K}) g(\ol{K}) ) h(\ol{L}) \\
                                    &= \sum_{L \subset N \subset M} f(\ol{M/N}) g(\ol{N/L}) h(\ol{L}),
\end{align*}
where in the last step we have used the fact (see property \ref{property9} of section  \ref{Amodules} ) that there is an inclusion-preserving bijection between sub-modules $K \subset M / L$ and sub-modules $N \subset M$ containing $L$, under which $N / L \simeq K$.  This bijection is compatible with taking quotients in the sense that $(M/L)/K \simeq M/N$. 
\end{proof}

\bigskip

$\H_{\A}$ is spanned by $\delta$-functions $\delta_{\ol{M}} \in \H_A$ supported on individual isomorphism classes, and so it is useful to make explicit the multiplication of two such elements. We have
\begin{equation} \label{deltamult}
\delta_{\ol{M}} \star \delta_{\ol{N}} = \sum_{ \ol{R} \in \on{Iso}(\C^N_A) } \mathbb{P}^R_{M,N} \delta_{\ol{R}} 
\end{equation}
where 
\[
\mathbb{P}^R_{M,N}  := \# \vert \{ L \subset R, L \simeq N, R/L \simeq M \} \vert
\]
As explained in the introduction, 
\[
\mathbb{P}^R_{M,N} \vert \on{Aut}(M) \vert \vert \on{Aut}(N) \vert
\]
counts the isomorphism classes of admissible short exact sequences of the form
\begin{equation} \label{ses}
0 \rightarrow N \rightarrow R \rightarrow M \rightarrow 0,
\end{equation}
where $ \on{Aut}(M) $ is the automorphism group of $M$. 

\bigskip

We may also equip $\H_{\A}$ with a coproduct
\begin{equation*}
\Delta: \H_{\A} \rightarrow \H_{\A} \otimes \H_{\A}
\end{equation*}
given by
$ 
\Delta(f)(\ol{M},\ol{N}) := f(\ol{M \oplus N}).
$ 
The coproduct $\Delta$ is clearly co-commutative.

\begin{lemma}
The following holds in $\H_{\A}$:
\begin{enumerate}
\item $\Delta$ is co-associative: $(\Delta \otimes \on{Id}) \circ \Delta = (\on{Id} \otimes \Delta) \circ \Delta$
\item $\Delta$ is compatible with $\star$:  $\Delta(f \star g) = \Delta(f) \star \Delta(g).$
\end{enumerate}
\end{lemma}

\begin{proof}
The proof of both parts is the same as the proof of the corresponding statements for the Hall algebra of the category of quiver representations over $\fun$, given in \cite{Sz2}. 
\end{proof}

$\H_{\A}$ carries a natural grading by $\mathbb{Z}_{\geq 0}$ corresponding to the dimension of $M \in \C^N_{\A}$.  
With this grading, $\H_{\A}$ becomes a graded, connected, co-commutative bialgebra, and thus automatically a Hopf algebra. By the Milnor-Moore Theorem, $\H_{\A}$ is isomorphic to  $\mathbb{U}(\n_{\A})$ - the universal enveloping algebra of $\n_{\A}$, where the latter is the Lie algebra of its primitive elements.  and the definition of the co-product implies that $\n_{\A}$ is spanned  by $\delta_{\ol{M}}$ for isomorphism classes $\ol{M}$ of indecomposable $\A$--modules, with bracket
\[
[\delta_{\ol{M}}, \delta_{\ol{N}}] = \delta_{\ol{M}} \star \delta_{\ol{N}} - \delta_{\ol{N}} \star \delta_{\ol{M}}.
\]
We have thus proved:

\begin{theorem}
Let $\A$ be a finitely generated semigroup. There is an isomorphism of Hopf algebras $\H_{\A} \simeq \mathbb{U}(\n_{\A})$, where $\H_{\A}$ is the Hall algebra of the category $\C^N_{\A}$ of finite $\A$--modules with normal morphisms, and $\n_{\A}$ is the Lie sub-algebra spanned by $\delta_{\ol{M}}$ for $M$ indecomposable.
\end{theorem}

\section{Examples} \label{examples}
\subsection{The monoid $\fm$ and linear algebra over $\fun$} \label{fm}

Let $\fm$ denote the free monoid on one generator, i.e. $\fm := \{ 0, 1, t, t^2, t^3 \cdots \}$. Given a $\fm$--module $M$, the action of the generator $t$ yields a map of pointed sets ($\fun$ vector spaces) $t: M \rightarrow M$, and conversely giving such a map equips $M$ with a $\fm$--module structure. For a field $k$, linear algebra over $k$ is the study of modules over the polynomial ring $k[t]$, which is the base-change $\fm \otimes_{\fun} k$. Thus we may view the study of $\fm$--modules as linear algebra over $\fun$. 

Given $M \in \C^N_{\fm}$, we may attach to it an oriented graph $\Gamma_M$, with $V(\Gamma_M) := M \backslash \{ * \}$ (i.e. the non-zero elements of $M$), and the oriented edges $$E(\Gamma_M) :=  \{ m \rightarrow t \cdot m \vert m \in M, m \neq * \}.$$ Every vertex in  $\Gamma_M$ therefore has at most one out-going edge. We have $\Gamma_{M \oplus N} = \Gamma_M \coprod \Gamma_N$, and it follows from this that the connected components of $\Gamma_M$ correspond to its indecomposable factors. We proceed to characterize the possible graphs $\Gamma_M$. Suppose that $M$ is indecomposable. There are two possibilities for the action of $t$ on $M$:

\begin{enumerate}
\item $\exists n \in \mathbb{N}$ such that $t^n \cdot m = * \; \forall m \in M$ - in this case we say that $M$ and $t$ are \emph{nilpotent}
\item $\exists m \in M$ such that $t^n \cdot m \neq * \; \forall n \in \mathbb{N}$. Since $\{ t^k\cdot m \}_{k \in \mathbb{N}} \subset M$ is finite, this implies that there are $n_1, n_2 \in \mathbb{N}$ such that $t^{n_1} \cdot m = t^{n_2} \cdot m$.
\end{enumerate}

Thus, in the first case, $\Gamma_M$ is a tree, and in the second, it contains an oriented cycle. 

\begin{example}  \label{graphs}
Examples of $\Gamma_M$:
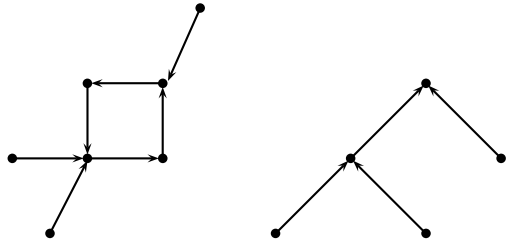
\begin{center}          
                 
\begin{pspicture}(7,3)

\psdots(1.5, 1)
\psdots(2.5, 1)
\psdots(2.5, 2)
\psdots(1.5, 2)
\psdots(1,0)
\psdots(0.5,1)
\psdots(3,3)

\psdots(4,0)
\psdots(6,0)
\psdots(5,1)
\psdots(7,1)
\psdots(6,2)

\psline{->}(4,0)(4.97,0.97)
\psline{->}(6,0)(5.03,0.97)
\psline{->}(5,1)(5.97,1.97)
\psline{->}(7,1)(6.03,1.97)

\psline{->}(3,3)(2.57,2.03)
\psline{->}(1.5, 1)(2.45, 1)
\psline{->}(2.5, 1)(2.5, 1.95)
\psline{->}(2.5, 2)(1.55, 2)
\psline{->}(1.5, 2)(1.5, 1.05)
\psline{->}(1,0)(1.5, 0.97)
\psline{->}(0.5,1)(1.45,1)
\end{pspicture}

\end{center}

\end{example}

If $\Gamma$ is a (not necessarily oriented) graph, we may describe a path $\sigma$ in $\Gamma$ by the ordered tuple $[v_1, v_2, \cdots, v_k]$ of consecutive vertices $v_i \in \Gamma$ encountered along the way (i.e. in this notation, $v_1$ is the starting vertex and $v_k$ the final one). The opposite path $[v_k, \cdots, v_1]$ will be denoted $\overline{\sigma}$.  A cycle $[v_1, v_2, \cdots, v_k, v_1] \; k \geq 3$ will be called \emph{minimal} if $v_i \neq v_j$ for $1 \leq i \neq j \leq k$.  Given a vertex $v \in \Gamma_M$ with $a$ incoming and $b$ outgoing edges, we will call the ordered pair $(a,b)$ its \emph{type}. Since each vertex has at most one outgoing edge, the possible types are $(a,0)$ and $(a, 1)$, $a \in \mathbb{Z}_{\geq 0}$. Note that a vertex of type $(a,0)$ corresponds to an element $m \in M$ such that $t \cdot m = *$. We will denote by $\gamma_M$ the un-oriented graph underlying $\Gamma_M$, and by $h_1 (\gamma_M) = h_1 (\Gamma_M)$ their first betti number. A path in $\gamma_M$ is said to be \emph{correctly oriented} if it arises from an oriented path in $\Gamma_M$.

\begin{lemma} \label{single_root}
If $\Gamma_M$ is connected (hence $M$ indecomposable), it contains at most one vertex of type $(a,0)$.
\end{lemma}

\begin{proof}
Suppose that $u, w$ are two vertices of type $(a,0)$, and $\sigma = [u, v_1, \cdots, v_k, w]$ is a path in $\gamma_M$ from $u$ to $w$, which we may assume to contain no cycles (i.e. each vertex occurs once). Note that neither $\sigma$ nor $\overline{\sigma}$ arise from oriented paths in $\Gamma_M$, since their initial edges are traversed in a direction opposite to their orientation in $\Gamma_M$. It follows that at least one of $v_1, \cdots, v_k$ must have at least two out-going edges in $\Gamma_M$, yielding a contradiction. 
\end{proof}

\bigskip

\begin{lemma}
If $\sigma = [v_1, v_2, \cdots, v_k, v_1]$ is a minimal cycle in $\gamma_M$, then either $\sigma$ or $\overline{\sigma}$ is correctly oriented. 
\end{lemma}

\begin{proof}
Suppose not. Then one of the vertices $v_i$ must have at least two incoming or out-going edges. It is easy to see a vertex with at least two incoming edges forces the existence of another one with at least two out-going edges, yielding a contradiction. 
\end{proof}

\bigskip

\begin{lemma}
If $h_1(\Gamma_M) > 0$, then $\Gamma_M$ contains exactly one oriented cycle (i.e. $h_1 (\Gamma_M) = 1$). 
\end{lemma}

\begin{proof}
Suppose $h_1 (\Gamma_M) >1$. Then by the previous lemma we may find two distinct oriented cycles $\sigma_1, \sigma_2$. These cannot share a vertex, since an oriented cycle is determined by any of its vertices. We may thus find pair of vertices $u \in \sigma_1, w \in \sigma_2$, and a path $\tau=[u, v_1, \cdots v_k, w]$ in $\gamma_M$ connecting $u$ to $w$. Moreover, $\tau$ may be chosen free of cycles, and such that $v_1, \cdots, v_k$ are disjoint from $\sigma_1, \sigma_2$. It is clear that $\tau$ is not correctly oriented since $[u, v_1]$ is not but $[v_k,w]$ is. This forces the existence of a vertex in $\tau$ with two out-going edges, yielding a contradiction.  
\end{proof}

\bigskip

Suppose now that $h_1 (\Gamma) = 1$, and let $\sigma$ be the unique oriented cycle. Denote by $\Gamma_M / \sigma$ (resp. $\gamma_M / \sigma$) the directed (resp. un-directed ) graph obtained by collapsing $\sigma$ to a point, and by $r$ the vertex of $\Gamma_M / \sigma$ and $\gamma_M / \sigma$ corresponding to the collapsed $\sigma$. It follows from the previous lemma that $h_1 (\Gamma_M / \sigma) = 0$, so that $\Gamma_M / \sigma$ is a tree. This shows that in $\gamma_M$, there is a unique shortest path $\tau_{v \sigma} = [v, v_1, \cdots, v_k, w]$ from any vertex $v \in \Gamma_M \backslash \sigma$ to the cycle $\sigma$, having the property that $w \in \sigma$, but $v_1, \cdots, v_k \in \Gamma_M \backslash \sigma$. I claim that $\tau_{v \sigma}$ is correctly oriented. This follows since $[v_k w]$ is correctly oriented, and so if $\tau_{v \sigma}$ is not, one of $v_1, \cdots, v_k$ would have to have at least two outgoing edges. Thus, $\Gamma_M / \sigma$ is canonically a rooted tree, with root $r$, and this shows that $\Gamma_M$ is obtained by attaching rooted trees $T_1, \cdots, T_l$ to vertices of the oriented cycle $\sigma$ at the roots of $T_i$. In particular, $\Gamma_M$ contains no vertices of type $(1,0)$.

To summarize:

\begin{proposition} \label{gammaclass}
Suppose that $M$ is indecomposable, and hence $\Gamma_M$ connected. 
\begin{enumerate}
\item If $M$ is nilpotent, then $\Gamma_M$ is a rooted tree.
\item If $M$ is not nilpotent, then $\Gamma_M$ is obtained by attaching rooted trees to the vertices of a unique oriented cycle $\sigma$. 
\end{enumerate}
\end{proposition}

\bigskip

We proceed to describe the $\fm$--submodules of a $\fm$ module $M$ in terms of $\Gamma_M$. By Remark \ref{subobj}, it suffices to describe the sub-modules of $M$ when $M$ is indecomposable, hence $\Gamma_M$ connected. Such an $N$ then corresponds to an oriented sub-graph $\Gamma_N \subset \Gamma_M$ with the property that any oriented path in $\Gamma_M$ starting at a vertex of $\Gamma_N$ stays in $\Gamma_N$ - this clearly being equivalent to the condition that $N$ is $\fm$--invariant. We call such a $\Gamma_N$ \emph{invariant}.  

\medskip

\begin{definition}
An \emph{admissible cut} on an oriented graph $\Gamma$ is collection $\Phi \subset E(\Gamma)$ of edges of $\Gamma$ such that at most one member of $\Phi$ is encountered at most once along any oriented path in $\Gamma$. An admissible cut is called \emph{simple} if $\Phi$ consists of a single edge. 
\end{definition}

\medskip

\begin{example}
The following are examples of admissible cuts on the graphs from Example \ref{graphs}. The cut edges are indicated by dashed lines. 

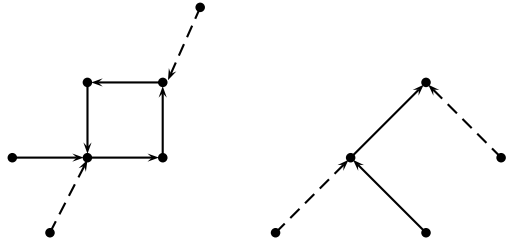
\begin{center}          
                 
\begin{pspicture}(7,3)

\psdots(1.5, 1)
\psdots(2.5, 1)
\psdots(2.5, 2)
\psdots(1.5, 2)
\psdots(1,0)
\psdots(0.5,1)
\psdots(3,3)

\psdots(4,0)
\psdots(6,0)
\psdots(5,1)
\psdots(7,1)
\psdots(6,2)

\psline[linestyle=dashed]{->}(4,0)(4.97,0.97)
\psline{->}(6,0)(5.03,0.97)
\psline{->}(5,1)(5.97,1.97)
\psline[linestyle=dashed]{->}(7,1)(6.03,1.97)

\psline[linestyle=dashed]{->}(3,3)(2.57,2.03)
\psline{->}(1.5, 1)(2.45, 1)
\psline{->}(2.5, 1)(2.5, 1.95)
\psline{->}(2.5, 2)(1.55, 2)
\psline{->}(1.5, 2)(1.5, 1.05)
\psline[linestyle=dashed]{->}(1,0)(1.5, 0.97)
\psline{->}(0.5,1)(1.45,1)
\end{pspicture}

\end{center}

\end{example}

\medskip

\begin{remark} \label{nocycleedges}
It follows immediately that an admissible cut may not include any edges that lie along an oriented cycle. 
\end{remark}

\medskip

\begin{lemma}
If $M$ is indecomposable, then sub-modules $N \subset M$ correspond to admissible cuts on $\Gamma_M$.
 \end{lemma}

\begin{proof}
Suppose $N \subset M$, and let $\Phi_N \subset E(\Gamma_M)$ be the collection of edges joining a vertex of $\Gamma_M \backslash \Gamma_N$ to one in $\Gamma_N$. I claim that $\Phi_N$ is an admissible cut. If not, then there exists an oriented path $\sigma$ in $M$ and two distinct edges $e_1, e_2 \in \Phi_N$ that lie along it, where $e_1$ occurs before $e_2$. Let $\sigma' \subset \sigma$ be the sub-path of $\sigma$ starting at  $t(e_1)$ and ending at $s(e_2)$. The existence of $\sigma'$ contradicts the fact that $\Gamma_N$ is invariant. 

Conversely, suppose that $\Phi \subset E(\Gamma_M)$ is an admissible cut, and $\Gamma'_M$ the oriented graph obtained from $\Gamma_M$ by removing the edges in $\Phi$. Let $\Gamma_N$ be the connected component of $\Gamma'_M$ containing the root (if $\Gamma_M$ is a tree), or a point on the cycle (if $\Gamma_M$ contains an oriented cycle). Then it is clear that $\Gamma_N$ is an invariant sub-graph. 
\end{proof}

\medskip

\begin{remark} \label{simplecuts}
It follows from Remark \ref{nocycleedges} and Proposition \ref{gammaclass} that $\Gamma'_M$ consists of two connected components exactly when the admissible cut $\Phi$ is simple. In this case, denote by $Rt_{\Phi}(M)$ the component of $\Gamma'_M$ containing either the root or cycle, and by $Lf_{\Phi}(M)$ the remaining component. 
\end{remark}

\subsubsection{$\H_{\fm}$ and Kreimer's Hopf algebra of rooted trees}

We may now give a simple description of the Hall algebra $\H_{\fm}$ in terms of the combinatorics of the graphs $\Gamma_M$. Recall that $\H_{\fm} \simeq \mathbb{U}(\n_{\fm})$, where $\n_{\fm}$ is the Lie algebra spanned by $\delta_{\ol{M}}$ for indecomposable $\fm$--modules $M$. Combining equation \ref{deltamult}  and Remark \ref{simplecuts} we obtain:
\begin{equation} \label{treeprod}
\delta_{\ol{M}} \star \delta_{\ol{N}} = \delta_{\ol{M \oplus N}} + \sum_{\ol{R} \in \on{IndMod}} n(R,M,N) \delta_{\ol{R}}
\end{equation}
where $n(R,M,N)$ denotes the number of simple cuts $\Phi$ on $\Gamma_R$ such that $Rt_{\Phi}(R) \simeq \Gamma_N$ and $Lf_{\Phi} \simeq \Gamma_M$, and $\on{IndMod}$ denotes the set of isomorphism classes of indecomposable modules. 

The nilpotent $\fm$--modules form a full subcategory of $\C^N_{\fm}$, closed under extensions, which we denote by $\C^N_{\fm, nil}$. We may therefore consider the Hall algebra $\H_{\fm,nil}$ of $\C^N_{\fm, nil}$. If $M$ is nilpotent, $\Gamma_M$ is a disjoint union of rooted trees, i.e. a \emph{rooted forest}. Rooted forests appeared in the work of Dirk Kreimer as the algebraic backbone of perturbative renormalization in quantum field theory (see \cite{Kre}). In particular, it was shown in \cite{Kre} that they carry a natural structure of a commutative (but not co-commutative) Hopf algebra $\H_{K}$. By the Milnor-Moore theorem, $\H^*_{K}$ is seen to be an enveloping algebra, and comparing the formula \ref{treeprod} with the combinatorial description of the Lie bracket in \cite{Kre}, immediately yields the following:

\begin{theorem}
$\H^*_{K}$ and $\H_{\fm,nil}$ are isomorphic as Hopf algebras. 
\end{theorem}

\bigskip

\begin{remark}
It is a classical result (see \cite{S}) that when $\mc{A}$ is the category of finite-dimensional nilpotent $\mathbb{F}_q [t]$--modules, $\H_{\mc{A}}$ is isomorphic to the Hopf algebra of symmetric functions, which we denote by $\Lambda$. As an algebra, $$\Lambda \simeq \mathbb{Q}[e_i], \; i \in \mathbb{N}, $$ where the collection of $e_i$ may be taken to be a variety of symmetric polynomials (for instance, elementary symmetric functions, power sums etc.). From this perspective, rooted forests can be viewed as the $\fun$--analogue of symmetric functions. It is known furthermore \cite{F, Ch} that $\H^{*}_{K}$ is a free non-commutative algebra.  
\end{remark}

\subsubsection{ $\on{Rep}^{\wedge}[\fm], \on{Rep}^{\otimes}[\fm]$ and combinatorial operations on forests}

Given $M, N$ in $\C^N_{\fm}$, it is natural to ask how the graphs $\Gamma_{\ol{M \wedge N}}, \Gamma_{\ol{M \otimes N}}$ are related to $\Gamma_{\ol{M}}$ and $\Gamma_{N}$, i.e. for an explicit description of the ring structure in $\on{Rep}^{\wedge}[\fm], \on{Rep}^{\otimes}[\fm]$. In particular, since the subcategory $\C^N_{\fm, nil}$ is closed under all three operations, we obtain two commutative combinatorial operations on rooted forests. As the examples below suggest, these are non-trivial: 
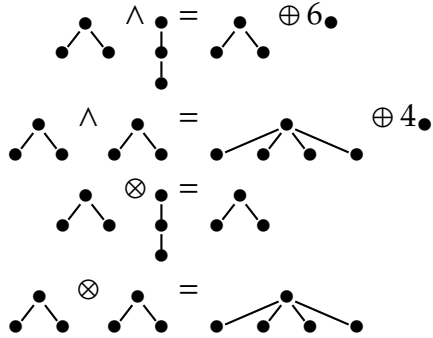
\begin{align*}
 \psset{levelsep=0.4cm, treesep=0.4cm}
 \pstree{\Tr{\bullet}}{\Tr{\bullet}\Tr{\bullet}} \wedge  \pstree{\Tr{\bullet}}{\pstree{\Tr{\bullet}}{\Tr{\bullet}}} &= \psset{levelsep=0.4cm, treesep=0.4cm}
   \pstree{\Tr{\bullet}}{\Tr{\bullet}\Tr{\bullet}} \oplus 6  \Tr{\bullet}   \\
 \psset{levelsep=0.4cm, treesep=0.4cm}
 \pstree{\Tr{\bullet}}{\Tr{\bullet}\Tr{\bullet}} \wedge  \pstree{\Tr{\bullet}}{\Tr{\bullet}\Tr{\bullet}} &=  \psset{levelsep=0.4cm, treesep=0.4cm}
  \pstree{\Tr{\bullet}}{\Tr{\bullet}\Tr{\bullet}\Tr{\bullet} \Tr{\bullet}} \oplus  4 \Tr{\bullet} \\
   \psset{levelsep=0.4cm, treesep=0.4cm}
 \pstree{\Tr{\bullet}}{\Tr{\bullet}\Tr{\bullet}} \otimes  \pstree{\Tr{\bullet}}{\pstree{\Tr{\bullet}}{\Tr{\bullet}}} &= \psset{levelsep=0.4cm, treesep=0.4cm}
   \pstree{\Tr{\bullet}}{\Tr{\bullet}\Tr{\bullet}}   \\
 \psset{levelsep=0.4cm, treesep=0.4cm}
 \pstree{\Tr{\bullet}}{\Tr{\bullet}\Tr{\bullet}} \otimes \pstree{\Tr{\bullet}}{\Tr{\bullet}\Tr{\bullet}} &=  \psset{levelsep=0.4cm, treesep=0.4cm}
  \pstree{\Tr{\bullet}}{\Tr{\bullet}\Tr{\bullet}\Tr{\bullet} \Tr{\bullet}} 
\end{align*}


\subsection{The monoids $\fm/\sim$ } \label{fm_cong}

Recall from Example \ref{cong_ex} that for any $x \in \fm, n \in \mathbb{N}$, the equivalence relation generated by $t^{k+n} \sim t^k x, k \geq 0$ is a congruence. To see that these are all possible congruences on $\fm$, observe that $\A = \fm / \sim$ is naturally a $\fm$-module, generated by $\ol{1}$. $\Gamma_{\fm/\sim}$ therefore has at most one vertex of type $(0,1)$. It follows from the classification of possible $\Gamma$'s above that $\Gamma_{\fm/\sim}$ is either a ladder tree (if $x = 0$), or a cycle with a single (possibly empty) ladder tree attached (if $x \neq 0$). 

$\fm$ maps surjectively to $\A$, and so any $\A$-module is automatically a $\fm$-module ( which has to respect $\sim$). We can thus use the classification of graphs $\Gamma$ above to describe the Lie algebra $\n_{\A}$. For this, we need the notion of \emph{height} of a rooted tree:

\begin{definition}
The \emph{height} of a rooted tree $T$ is the length of the longest path from leaf to root. 
\end{definition}

We distinguish two cases:

\begin{enumerate}
\item {\bf $x=0$}. Then $\ol{t}^n=0$ in $\A$, and so $\ol{t}$ acts nilpotently on any module. If $M$ is an indecomposable $\A$-module, then $\Gamma_M$ is a rooted tree of height $\leq n-1$ and conversely, any such rooted tree corresponds to an indecomposable $\A$--module. 
\item {\bf $x = t^m, \; m \geq 0$} We may assume that $m < n$ (if $m=n$ we get the identity equivalence relation, which was treated in the previous section). If $\ol{t}$ acts nilpotently on an indecomposable module $M$, then the condition $\ol{t}^n = \ol{t}^m$ implies that $\Gamma_M$ is a rooted tree of height at most $n-m-1$. If the action of $\ol{t}$ is not nilpotent, then for any non-zero $x \in M$, $\ol{t}^m x$ must be part of a cycle of length dividing $n-m$. Thus, the possible $\Gamma_M$ are either rooted trees of height at most $n-m-1$ or cycles of length dividing $n-m$ with rooted trees of height at most $m$ attached at the roots. 
\end{enumerate}

\subsection{Quiver representations over $\fun$}

Recall that a quiver $Q$ is a directed graph (which we will assume to be finite). To $Q$ we may attach a finitely generated semigroup $\A(Q)$, with generators $0$, $v_i, i \in V(Q)$ and $e_l, l \in E(Q)$, and relations
\begin{align*}
v_i v_j &= \delta_{i,j} v_i \\
e_l v_i &= 0 \textrm{ unless $i$ is the terminal vertex of $l$, in which case } e_l v_i = e_l \\
v_i e_l &= 0 \textrm{ unless $i$ is the initial vertex of $l$, in which case } v_i e_l  = e_l \\
0 \cdot x &= 0 \textrm{ for any element } x 
\end{align*} 
Informally, the non-zero elements of $\A(Q)$ consist of paths in $Q$ (including the trivial paths $v_i$ corresponding to vertices of $Q$), with multiplication given by concatenation of paths when it makes sense, and $0$ otherwise. Note that $\A(Q)$ does not in general have a unit.

Let $Q_J$ denote the Jordan quiver, consisting of a single vertex $v$ and a single edge (loop) $e$ from $v$ to $v$. We then have $\A(Q_J) \simeq \fm$, and so $\H_{\A(Q_J)}$ is a generalization of $\H_{\fm}$ above. A detailed description of $\H_{\A(Q_J)}$ will be given in \cite{Sz4}. 

\subsection{The monoid $\overline{G}$ and the Burnside ring} \label{G}

Let $G$ be a finite group, and $\Gb = G \cup \{ 0 \}$ the monoid obtained by adjoining to it the absorbing element $0$. We proceed to describe the category $\C^N_{\Gb}$ and its Hall algebra. Let $M$ be a $\Gb$--module, and $M' = M \backslash \{ * \}$ the set obtained by removing the base-point. $M'$ carries an action of $G$, since every element of $G$ has an inverse. It follows that every $\Gb$--module arises from a $G$--set by adjoining a base-point. $M'$ decomposes into a disjoint union $M'=\coprod^k_{i=1} M'_i$ of $G$--orbits, and setting $M_i = M'_i \cup \{ * \}$, we see that $M = \oplus^k_{i=1} M_i$ is the unique (up to permutation) decomposition of $M$ into irreducible hence indecomposable factors. Each orbit $M_i$ is of the form $G/H$ for a subgroup $H \subset G$, and since conjugate subgroups produce isomorphic $\Gb$--modules, we see that the non-trivial irreducible $\Gb$--modules are in bijection with conjugacy classes of subgroups of $G$. The action of $G$ on each orbit is transitive, so the notions of indecomposable and irreducible module are equivalent. 

Let $\mc{C}onj(G)$ denote set of conjugacy classes of subgroups of $G$, and for $i \in \mc{C}onj(G)$, denote by $H_i$ (any) subgroup belonging to the corresponding conjugacy class. Let $M_i = G/H_i \cup \{ * \}$, viewed as a left $\Gb$--module. 

If $N \subset M$ is a $\Gb$--submodule, then defining $N', M'$ as above, and setting $K = M \backslash N'$, we see that $M = N \oplus K$, thus the category $\C^N_{\Gb}$ is semi-simple in the appropriate sense. It follows that every admissible short exact sequence splits, and so if $M$ and $N$ are distinct indecomposable $\Gb$--modules, then in $\H_{\Gb}$, 
\[
\delta_{\ol{M}} \star \delta_{\ol{N}} = \delta_{\ol{M \oplus N}}.
\]
It follows that $\H_{\Gb}$ is free commutative, with generators $\ol{M}_i, \; i \in \mc{C}onj(G)$, and $\n_{\Gb}$ is abelian.

The ring $\on{Rep}^{\wedge}[\Gb]$ is also a free module on the $\ol{M}_i$. If $M_i \simeq G/H_i \cup \{ *\}$ and $M_j \simeq G / H_j \cup \{ * \}$ then $M_i \wedge M_j \simeq G/H_i \times G/H_j \cup \{*\}$, with $\Gb$ acting diagonally. It is therefore exactly the Burnside ring $\Omega(G)$ of G (see \cite{Sol}). $\on{Rep}^{\otimes}[\Gb]$ is not defined unless $G$ is abelian.  

\begin{proposition} Let $G$ be a finite group, and $\Gb$ the monoid $G \cup \{ 0 \}$. 
\begin{enumerate}
\item The category $\C^N_{\Gb}$ is semi-simple, in the sense that any $M \in \C^N_{\Gb}$ can be written uniquely (up to permutation) $$M \simeq \oplus^k_{j=1} M_{i_j}, \; i_j \in \mc{C}onj(G). $$
\item $\H_{\Gb} \simeq \mathbb{Q}[\ol{M}_i], \; i \in \mc{C}onj(G) $.
\item $\on{Rep}^{\wedge}[\Gb] \simeq \Omega(G)$, where $\Omega(G)$ is the Burnside ring of G. 
\end{enumerate}
\end{proposition}

\end{document}